\documentclass[12pt,a4paper]{amsart}
\usepackage[utf8]{inputenc}
\usepackage[T1]{fontenc}
\usepackage{mathtools}
\usepackage{amsmath}
\usepackage{amsthm}
\usepackage{amssymb}
\usepackage[abbrev]{amsrefs}
\usepackage{mathrsfs}
\usepackage[dvipsnames]{xcolor}
\usepackage{bm}
\usepackage{enumitem}
\usepackage{hyperref}
\usepackage{graphicx}
\usepackage[all]{xy}
\usepackage{ifthen}
\AtBeginDocument{\def\MR#1{}}
\makeatletter
\@namedef{subjclassname@2020}{%
\textup{2020} Mathematics Subject Classification}
\makeatother
\numberwithin{equation}{section}
\allowdisplaybreaks
\newtheorem{Question}{}

\newtheorem{question}[Question]{Question}
\newtheorem{Claim}{}

\newtheorem{claim}[Claim]{Claim}
\newtheorem{thm}{}[section]
\newtheorem{theorem}[thm]{Theorem}

\newtheorem{lemma}[thm]{Lemma}
\newtheorem{proposition}[thm]{Proposition}
\theoremstyle{definition}

\newtheorem{definition}[thm]{Definition}
\newtheorem{remark}[thm]{Remark}

\newcommand{\xx}{\ensuremath{\bm{x}}}
\newcommand{\SL}{\ensuremath{\mathscr{L}}}
\newcommand{\Id}{\ensuremath{\mathrm{Id}}}
\newcommand{\WW}{\ensuremath{\mathbb{W}}}
\newcommand{\XX}{\ensuremath{\mathbb{X}}}
\newcommand{\NN}{\ensuremath{\mathbb{N}}}
\newcommand{\RR}{\ensuremath{\mathbb{R}}}

\newcommand{\ZZ}{\ensuremath{\mathbb{Z}}}
\newcommand{\VV}{\ensuremath{\mathbb{V}}}
\newcommand{\LL}{\ensuremath{\mathbb{L}}}
\newcommand{\YY}{\ensuremath{\mathbb{Y}}}
\newcommand{\Pt}{\ensuremath{\mathcal{P}}}
\newcommand{\Mt}{\ensuremath{\mathcal{M}}}
\newcommand{\Kt}{\ensuremath{\mathcal{K}}}
\newcommand{\Ut}{\ensuremath{\mathcal{U}}}
\newcommand{\It}{\ensuremath{\mathcal{I}}}
\newcommand{\Nt}{\ensuremath{\mathcal{N}}}
\newcommand{\Vt}{\ensuremath{\mathcal{V}}}
\newcommand{\Jt}{\ensuremath{\mathcal{J}}}
\newcommand{\F}{\ensuremath{\mathcal{F}}}
\newcommand{\At}{\ensuremath{\mathcal{A}}}
\newcommand{\XB}{\ensuremath{\mathcal{X}}}
\newcommand{\Rt}{\ensuremath{\mathcal{R}}}
\newcommand{\ww}{\ensuremath{{\bm{w}}}}
\newcommand{\EC}{\ensuremath{{\bm{A}}}}
\DeclareMathOperator{\leaf}{leaf}
\DeclareMathOperator{\diam}{diam}
\DeclareMathOperator{\Lip}{Lip}
\newcommand{\Ker}{\operatorname{\rm{Ker}}}
\newcommand{\abs}[1]{\left\lvert#1\right\rvert}
\newcommand{\norm}[1]{\left\lVert#1\right\rVert}
\newcommand{\enfloor}[1]{\left\lfloor#1\right\rfloor}
\newcommand{\enceil}[1]{\left\lceil#1\right\rceil}
\newcommand{\enbrace}[1]{\left\lbrace#1\right\rbrace}
\newcommand{\enbrak}[1]{\left[#1\right]}
\newcommand{\enpar}[1]{\left(#1\right)}
\newcommand{\Env}[2][]{\ifthenelse{\equal{#1}{}}{\ensuremath{#2_{\mathsf{c}}}}{\ensuremath{#2_{\mathsf{c},#1}}}}
\author[F. Albiac]{Fernando Albiac}
\address{Department of Mathematics, Statistics, and Computer Sciences--Ina\-Mat$^2$ \\
Universidad P\'ublica de Navarra\\
Campus de Arrosad\'{i}a\\
Pamplona\\
31006 Spain}
\email{fernando.albiac@unavarra.es}
\author[J. L. Ansorena]{Jos\'e L. Ansorena}
\address{Department of Mathematics and Computer Sciences\\
Universidad de La Rioja\\
Logro\~no\\
26004 Spain}
\email{joseluis.ansorena@unirioja.es}
\begin{document}
\title[Lipschitz extensions into $p$-Banach spaces]{Lipschitz extensions into $\bm p$-Banach spaces, and canonical embeddings of Lipschitz-free $\bm p$-spaces for $\bm {0<p<1}$}
\begin{abstract}
We show that inclusions of $p$-metric spaces always produce genuine linear embeddings at the level of Lipschitz-free $p$-spaces. More precisely, for every $0<p<1$ and every inclusion $\Nt\subset \Mt$ of $p$-metric spaces, the canonical map from $\F_p(\Nt)$ into $\F_p(\Mt)$ is always an isomorphic embedding, as it plainly happens for $p=1$. Our proof relies on a versatile extension procedure for $p$-Banach-valued Lipschitz maps, allowing us to control the geometry of canonical molecules and uncover a rigidity principle governing the structure of Lipschitz free $p$-spaces. As an application, we prove that, given $0<p<q\le 1$, the natural envelope map from the Lipschitz-free $p$-space $\F_p(\Mt)$ to its $q$-Banach envelope $\F_q(\Mt)$ is one-to-one. These results give positive answers to two foundational questions that were originally raised by Kalton in [Lipschitz structure of quasi-Banach spaces, Israel J. Math. 170 (2009), 317-335], and provide tools for furthering the understanding of subspace structures, hereditary properties, and geometric invariants in Lipschitz-free $p$-spaces.
\end{abstract}
\thanks{F.\@ Albiac and J.\@ L.\@ Ansorena acknowledge the support of the Spanish Ministry for Science and Innovation under Grant PID2022-138342NB-I00 for \emph{Functional Analysis Techniques in Approximation Theory and Applications}}
\maketitle
\section{Introduction}\label{sect:Intro}\noindent
The theory of Lipschitz-free spaces has become a central tool in nonlinear functional analysis, providing a canonical and highly flexible linearization of Lipschitz maps. In the classical setting of Banach spaces, the Lipschitz-free space $\F(\Mt)$ over a pointed metric space $\Mt$ has been thoroughly studied and is now understood as a rich source of geometric phenomena and a bridge connecting metric geometry, Banach space theory, and approximation theory. Indeed, this construction embeds metric subsets isometrically, it linearizes Lipschitz maps uniquely, and it yields a robust dictionary between Lipschitz and linear structures.

However, when the underlying space $\Mt$ is equipped with a $p$-metric for $0<p<1$ the situation changes dramatically: local convexity fails and many linear analytic tools break down. This makes the study of the Lipschitz-free $p$-spaces $\F_{p}(\Mt)$ conceptually appealing but technically challenging.

The construction of the spaces $\F_p(\Mt)$ was introduced by Kalton and the first-named author in their seminal paper \cite{AlbiacKalton2009}. Among several other foundational contributions to the nonlinear theory, these spaces were used to show that for every $0<p<1$ there exist separable $p$-Banach spaces that are Lipschitz-isomorphic but not linearly isomorphic, thereby revealing a deep and previously unseen gap between Lipschitz and linear classifications in the nonlocally convex regime. Whether this gap exists in the separable Banach space setting is one of the most important open questions in the field.

Despite this promising beginning, the structure of $\F_p(\Mt)$, for $\Mt$ a quasi-metric space, remained almost entirely unexplored for nearly a decade. A systematic investigation was initiated only recently in \cite{AACD2020}, where the authors developed the basic framework, clarified several subtle features of these spaces, and pointed out a series of fundamental obstacles that distinguish the quasi-Banach framework from the classical one. Their analysis revealed a recurrent theme: structural results that are automatic or routine in the Banach case often fail, or become unexpectedly delicate, when $0<p<1$. In particular, the lack of local convexity together with the scarcity of linear tools such as Hahn-Banach separation, duality, and convexity methods creates severe difficulties even for questions that appear elementary at first glance.

A fundamental instance of this phenomenon concerns the behavior of Lipschitz-free $p$-spaces with respect to inclusions. If $j\colon \Nt\to \Mt$ is an inclusion of pointed $p$-metric spaces, the universal linearization yields a canonical linear operator
\[
L[\Nt,\Mt;p]\colon \F_{p}(\Nt)\to \F_{p}(\Mt).
\]
In the locally convex case $p=1$, $L[\Nt,\Mt;p]$ is always an isometric embedding thanks to the validity of MacShane's theorem on the extension of real-valued Lipschitz maps on metric spaces. The problem of whether $L[\Nt,\Mt;p]$ remains an embedding when $\Mt$ is $p$-metric for $0<p<1$ was already highlighted in \cite{AlbiacKalton2009} (see the Remark after Lemma 4.1), and in the subsequent work \cite{AACD2020} it was showed that isometry can fail, thus motivating the more realistic version.

\begin{question}[\cite{AACD2020}*{Question 6.2}]\label{qt:AK}
Let $0<p<1$ and $\Nt \subset \Mt$ be $p$-metric spaces. Is the canonical linear map $L[\Nt,\Mt;p]$ an isomorphic embedding?
\end{question}

Despite the recent substantial progress on the structure of Lipschitz-free $p$-spaces over a series of papers \cites{AACD2020b,AACD2021,AACD2022,AACD2022b}, this question has remained elusive for several reasons. First, classical extension machinery (such as MacShane's extension of Lipschitz maps) breaks down in quasi-metric settings. Second, techniques relying on duality and convexity are unavailable in $p$-Banach spaces. Third, while the metric case enjoys a rich supply of linearization tools developed over several decades, the quasi-metric case requires fundamentally new methods.

B\'{\i}ma established in \cite{Bima2025} a powerful extension theory for Lipschitz maps with values in $p$-Banach spaces formulated in terms of absolute extendability constants. Using these techniques, C\'uth and Raunig \cite{CuthRaunig2024} made a decisive step forward towards the solution of Question~\ref{qt:AK}: they proved that if $\Nt\subset\Mt$ are metric spaces, then $L[\Nt,\Mt;p]$ is always an isomorphism (see \cite{CuthRaunig2024}*{Theorem 4.21}). Their theorem resolved the metric version of Question~\ref{qt:AK} and demonstrated that the obstacle lies not in the exponent $p<1$ but in the possible pathologies derived from quasi-metrics. In fact, their approach did not address the quasi-metric case of Question~\ref{qt:AK}, and no tools were available to bridge this final gap.

The present work completes this program by giving a positive solution to Question~\ref{qt:AK} in full generality. Our answer relies on showing that the functorial stability of Lipschitz-free spaces (a hallmark of the classical Banach case) persists throughout the nonlocally convex case, thus exhibiting that Lipschitz-free $p$-spaces faithfully reflect the metric structure of subsets just as classical Lipschitz-free spaces do. The fact that no additional regularity assumptions are needed is especially striking, given the lack of convexity.

The proof combines two ideas that, together, reveal a robustness in the Lipschitz-free $p$-space construction that had not been visible before.

\begin{itemize}[leftmargin=*]
\item\textbf{Reduction to metric trees.} Even if the geometry of quasi-metric spaces can be extremely convoluted, the possibility of extending Lipschitz maps defined on quasi-metric spaces only depends on the possibility of extending Lipschitz maps defined on certain metric trees associated with them. This finding by C\'uth and Raunig \cite{CuthRaunig2024} is related to the existence of a finite algorithm for evaluating Lipschitz-free space norms (see \cite{CuthRaunig2024}*{Theorem 2.2}).

\item\textbf{Extensions of Lipschitz maps through enlargements of the target space.}
Real-valued Lipschitz maps on metric spaces can be extended thanks to MacShane's theorem. In the case when the target space is a Banach space (or even a quasi-Banach space) extending Lipschitz maps defined on suitable metric spaces (such as metric trees) is still possible (see \cite{Bima2025}). Tackling the corresponding results for Lipschitz maps defined on quasi-metric spaces seems hopeless. Notwithstanding, such embeddings do exist if we replace the target space with a suitable space containing it!
\end{itemize}

Taken together, these ideas show that the quasi-metric case behaves far more regularly than might have been expected,
and encourage further study of the geometry of Lipschitz-free $p$-spaces.
In fact, not knowing whether the linearizations of inclusions are isomorphisms has hindered the development of the theory of Lipschitz-free $p$-spaces when $p<1$. Our positive answer to Question~\ref{qt:AK} will undoubtedly enable to overcome problems that had previously appeared insurmountable.

As a first illustration of the applicability of our results, we will solve another problem from \cite{AlbiacKalton2009} which has to do with the mappings connecting different spaces in the scale of Lipschitz-free spaces over a given quasi-metric space. If $0<p<q\le 1$ and $\Mt$ is a $q$-metric space, in particular $\Mt$ also is a $p$-metric space. Thus both $\F_p(\Mt)$ and $\F_q(\Mt)$ are $p$-Banach spaces and we can consider the canonical map
\[
J:=J_{p,q}[\Mt]\colon\F_p(\Mt) \to \F_q(\Mt).
\]
The nature of the elements of $\F_p(\Mt)$ and $\F_q(\Mt)$, which are defined through a completion procedure, makes it difficult to determine whether a given representation of an element in $\F_p(\Mt)$ results in the null vector when mapped into $\F_q(\Mt)$. Thus, in the face of this potentially ``pathological'' behavior, the following question naturally arises.

\begin{question}[see \cite{AACD2020}*{Remark 4.22}]\label{qt:AKEnv}
Let $0<p<q\le 1$ and $\Mt$ be a $q$-metric space. Is $J_{p,q}[\Mt]$ one-to-one?
\end{question}
Question~\ref{qt:AKEnv} has been recently solved in the metric case, that is, in the particular case when $q=1$ (see \cite{AlbiacAnsorena2026b}). In this paper, we solve it in full generality by a set of techniques, including our answer to Question~\ref{qt:AK}. In essence, our solution relies on two driving ideas.

\begin{itemize}[leftmargin=*]
\item\textbf{Reduction to Euclidean spaces}.
Despite the fact that quasi-metric spaces can look very different from Euclidean spaces, our study relies heavily on studying the geometry of the Lipschitz-free spaces over the latter. Specifically, while metric trees played a key role in solving Question~\ref{qt:AK}, to tackle Question~\ref{qt:AKEnv} we can narrow down the problem to studying nets in anti-snowflakings of Euclidean spaces. Recall that the $q$-anti-snowflaking of a metric space $(\Mt,d)$ for $0<q<1$, is the $q$-metric space $(\Mt,d^{1/q})$.

\item\textbf{Finite-dimensional decompositions}. It is often difficult to determine whether a given Lipschitz-free space has a Schauder basis or even the approximation property. In the non-locally convex setting, addressing this kind of problem presents additional difficulties. We overcome these drawbacks and prove that certain Lipschitz-free $p$-spaces over nets of anti-snowflakings of Euclidean spaces have a finite-dimensional Schauder decomposition.
\end{itemize}

We now provide an overview of the paper's structure. Section~\ref{sect:LipExt} is devoted to achieving extensions of Lipschitz maps. We show that the existence of Whitney covers allows us to extend Lipschitz maps from $p$-metric spaces into $p$-Banach spaces, $0<p\le 1$. Our construction exhibits the key role played by $L_p$-spaces in the study of Lipschitz-free $p$-spaces. Indeed, when extending a Lipschitz map into a $p$-Banach space $\XX$, we built a Lipschitz map into the vector-valued Lebesgue space $L_p(\XX)$. In Section~\ref{sect:CanEmb} we take advantage of the fact that metric trees have finite Nagata dimension, as well as the recent advances on linearizations of Lipschitz maps from \cite{CuthRaunig2024}, to solve Question~\ref{qt:AK}. Finally, in Section~\ref{sect:Envelope} we provide the machinery that permits us to answer Question~\ref{qt:AKEnv}.
\subsection*{Terminology}
We conclude this introductory section by establishing some notation. The symbol $\NN_0$ stands for the set of nonnegative integers. Depending on the context, $\abs{A}$ will denote the cardinality of a set $A$ or the Lebesgue measure of a measurable subset $A$ of an Euclidean space. Given a quasi-metric space $(\Mt,d)$, $E\subset\Mt$ and an interval $J\subset[0,\infty)$ we set
\[
\At(E;J)=\enbrace{x\in \Mt \colon d(x,E)\in J}.
\]
We denote by $B(x,t)$ the open ball of radius $t\in(0,\infty)$ centered at $x\in\Mt$, that is, $B(x,t)=\At(\enbrace{x},[0,t))$.

Given $0<q <\infty$, a measure space $(\Omega,\Sigma,\mu)$, and a quasi-Banach space $\XX$, $L_q(\mu,\XX)$ denotes the quasi-Banach space of all strongly measurable functions $f\colon\Omega\to \XX$ such that
\[
\norm{f}_{q,\XX} = \enpar{\int_\Omega \norm{f(\omega)}_\XX^q d\mu(\omega)}^{1/q}<\infty.
\]
In the case when $\mu$ is the Lebesgue measure on a set $A$, we put $L_q(\mu,\XX)=L_q(A,\XX)$. If $A$ is the unit interval, we set $L_q(A,\XX)=L_q(\XX)$.

If $0<p\le q$, the scalar-valued space $L_q(\mu)$ is a lattice $p$-convex quasi-Banach space with constant one, that is,
\[
\norm{\enpar{\abs{f}^p +\abs{g}^p}^{1/p}}_q \le\enpar{\norm{f}_q^p+\norm{g}_q^p}^{1/p}, \quad f,g\in L_q(\mu).
\]
Hence, if $0<p\le\min\{1,q\}$ and $\XX$ is a $p$-Banach space, then $L_q(\mu,\XX)$ is a $p$-Banach space.

For convenience, given $f\in L_q(\mu)$ and $x\in\XX$, we denote by $f\otimes x$ the function in $L_q(\mu,\XX)$ given by $f\otimes x(\omega)=f(\omega) x$ for all $\omega\in\Omega$ and $x\in\XX$. Note that
\[
\norm{f\otimes x}_{q,\XX}=\norm{f}_q \norm{x}_\XX.
\]
Similarly, given a function $x\colon \Mt \to\XX$, $f\otimes x$ stands for the obvious function $f\otimes x\colon \Mt\to L_p(\mu,\XX)$.

Given a finite set $J$ and $0<q \le \infty$, $\abs{x}_q$ will denote the $\ell_q$-norm of $x\in\RR^J$.
\section{Lipschitz extensions of mappings into \texorpdfstring{$p$}{}-Banach spaces}\label{sect:LipExt}\noindent
Keeping track of the quantitative estimates inherent to Whitney covers will be essential for us.
\begin{definition}\label{def:Whitney}
Let $V$ be an open subset of a metric space $(\Mt,\rho)$, $\kappa\in\NN$, $\alpha$, $\gamma\in[1,\infty)$, and $\beta\in(0,\infty)$. Assume that $\Nt:=\Mt\setminus V\neq\emptyset$. We say that a family $\Ut$ of nonempty open sets of $V$ is a \emph{Whitney cover of $V$ relative to $\Mt$ with parameters $(\kappa, \gamma, \beta, \alpha)$} if
\begin{enumerate}[label=(\alph*),leftmargin=*]
\item\label{it:W1} $\abs{\enbrace{U\in\Ut \colon x\in U}}\le \kappa$ for all $x\in V$,
\item\label{it:W2} for every $x\in V$ there exists $U\in\Ut$ such that $\rho (x,\Nt)\le \gamma \rho(x,\Mt\setminus U)$,
\item\label{it:W3} $\diam(U) \le \beta \rho(x,\Nt)$ for all $U\in\Ut$ and $x\in U$, and
\item\label{it:W4} $\rho(x,\Nt) \le \alpha \rho(y,\Nt)$ for all $U\in\Ut$ and all $x$, $y\in U$.
\end{enumerate}
\end{definition}

We emphasize that the parameters associated with Whitney covers that we use are slightly different from those in \cite{Bima2025}. Namely, we have replaced $\gamma$ with $1/\gamma$, so that the smaller $\kappa$, $\alpha$, $\beta$, $\gamma$, the more demanding the Whitney-type condition is. Condition~\ref{it:W2} implies $V\subset \cup_{U\in\Ut} U$. So, a Whitney cover of $V$ relative to a metric space is indeed a cover of $V$.

\begin{lemma}\label{lem:Bima}
Let $\Nt$ be a nonempty closed set of a metric space $(\Mt,\rho)$. Suppose that $V:=\Mt\setminus\Nt$, $\gamma\in[1,\infty)$, $\beta\in(0,\infty)$ and $\kappa\in\NN$, $\kappa\ge 2$, satisfy conditions \ref{it:W1}, \ref{it:W2} and \ref{it:W3} in the definition of a Whitney cover. Then there is a partition of unity $(\varphi_j)_{j\in\Jt}$ of $V$, $(x_j)_{j\in \Jt}$ in $\Mt \setminus V$ and constants $\mu=\mu(\kappa,\gamma)$, $\nu=\nu(\beta)$ such that
\begin{enumerate}[label=(B.\arabic*),leftmargin=*]
\item\label{it:B4} $\abs{\enbrace{j\in\Jt \colon \varphi_j(x)>0}}\le \kappa$ for all $x\in V$,
\item\label{it:B2} for all $j\in\Jt$ and all $x$, $y\in\Vt$ with $\varphi_j(x)>0$ and $\varphi_j(y)>0$,
\[
\rho(x,x_j) \le \nu \rho(y,\Nt),
\]
\item\label{it:B3} for all $x$, $y\in V$ such that $\min\{\varphi_j(x),\varphi_j(y)\}=0$ for all $j\in\Jt$,
\[
\rho(x,\Nt)\le \gamma \rho(x,y),
\]
\item\label{it:B1} and, for all $x$, $y\in V$,
\[
\sum_{j\in\Jt} \abs{\varphi_j(x)- \varphi_j(y)}
\le \mu \rho(x,y) \enpar{\frac{1}{\rho(x,\Nt)} + \frac{1}{\rho(y,\Nt)}}.
\]
\end{enumerate}
In fact, we can choose $\mu= 2 e \log(2) \gamma \log(2\kappa)$ and any $\nu>2+\beta$.
\end{lemma}

\begin{proof}
We exploit some ideas from \cite{Bima2025}*{Proof of Theorem~11} to improve the constants obtained there. Pick an increasing convex function $\psi\colon[0,\infty)\to[0,\infty)$ such that $\psi(x)=0$ if and only if $x=0$. For each $U\in\Ut$ define
\[
\psi_U\colon V \to [0,\infty), \quad x\mapsto \psi\enpar{d(x,\Mt\setminus U)},
\]
and $\Psi\colon V\to[0,\infty]$ by $\Psi=\sum_{U\in\Ut} \psi_U$. Given $x\in V$ and $U\in\Ut$, $\psi_U(x)>0$ if and only if $x\in U$. Hence, given $x\in V$, the set
\[
\enbrace{U\in\Ut \colon \psi_U(x)>0}=\enbrace{U\in\Ut \colon x\in U}
\]
is finite and nonempty, whence $0<\Psi(x)<\infty$. Therefore, if we put $\varphi_U=\psi_U/\Psi$ for all $U\in\Ut$, then $(\varphi_U)_{U\in\Ut}$ is a partition of unity. We will prove that it satisfies Properties~\ref{it:B4}, \ref{it:B2}, \ref{it:B3} and \ref{it:B1}, where the indexing set $\Jt$ is replaced with $\Ut$.

Since
\[
\enbrace{x\in V \colon \varphi_U(x)>0}=\enbrace{x\in V \colon \psi_U(x)>0}=U,
\]
\ref{it:B4} holds by \ref{it:W1}. To check \ref{it:B3}, we pick the set $U$ provided by \ref{it:W2}. Since $\varphi_U(x)>0$, $\varphi_U(y)=0$, that is, $y\notin U$. Hence,
\[
\rho (x,\Nt)\le \gamma \rho(x,\Mt\setminus U)\le \gamma\rho(x,y),
\]
as desired.

Use \ref{it:W3} to pick for each $U\in\Ut$ $x_U\in U$ such that
\[
\rho(x_U,\Nt) \le (\nu-1)\rho(U,\Nt)-\diam(U).
\]
For all $x$, $y\in U$ we have
\begin{align*}
\rho(x,x_U)
&\le \rho(x,y)+ \rho(y,\Nt)+\rho(x_U,\Nt) \\
&\le \diam(U)+ \rho(y,\Nt)+\rho(x_U,\Nt)\\
&\le \rho(y,\Nt) + (\nu-1) \rho(U,\Nt)\\
& \le \nu \rho(y,\Nt).
\end{align*}
We have proved that \ref{it:B2} holds. To check \ref{it:B1}, we set, for convenience, $a_U= d(x,\Mt\setminus U)$ and $b_U= d(y,\Mt\setminus U)$ for all $U\in\Ut$, and
\[
A=\enbrace{U\in\Ut \colon \min\enbrace{a_U,b_U}>0}.
\]
Define for $0\le t \le 1$
\[
G(t)=\sum_{U\in\Ut} \psi\enpar{(1-t)a_U+t b_U}, \quad H(t)=\sum_{U\in\Ut} \psi'\enpar{(1-t)a_U+t b_U}.
\]
Then, for $0\le t \le 1$ and $U\in\Ut$ we define
\[
h_U(t)=\frac{\psi\enpar{(1-t)a_U+t b_U}}{G(t)}.
\]
Finally, for $0\le t \le 1$ we set
\[
L(t)=\sum_{U\in\Ut} \abs{h_U'(t)}.
\]

In this terminology,
\begin{multline*}
D(x,y):=\sum_{U\in\Ut} \abs{\varphi_U(x)- \varphi_U(y)}
= \sum_{U\in\Ut} \abs{h_U(1)-h_U(0)}\\
=\sum_{U\in\Ut} \abs{\int_0^1 h_U'(t) \, dt} \le \int_0^1 L(t) \, dt.
\end{multline*}

Since
\begin{align*}
&\abs{h_U'(t)}
=\abs{\frac{\psi'\enpar{(1-t)a_U+t b_U}}{G(t)} - \frac{\psi\enpar{(1-t)a_U+t b_U}H(t)}{G^2(t)} } \abs{a_U-b_U}\\
&\le \max\enbrace{\frac{\psi'\enpar{(1-t)a_U+t b_U}}{G(t)}, \frac{\psi\enpar{(1-t)a_U+t b_U}H(t)}{G^2(t)} }\rho(x,y),
\end{align*}
we get
$
L(t)\le 2 \rho(x,y) {H(t)}/{G(t)}.
$
If $\psi$ is the power function $t\mapsto t^q$ for some $q\in[1,\infty)$, applying H\"older's inequality we obtain
\[
\frac{H(t)}{G(t)}
\le \frac{q \abs{A}^{1/q}}{\enpar{\sum_{U\in\Ut} \enpar{(1-t)a_U+t b_U}^q}^{1/q} }
\le \frac{q (2\kappa)^{1/q}}{\sup_{U\in\Ut} (1-t)a_U+t b_U }.
\]

Given $z\in\{x,y\}$, \ref{it:W2} gives $U\in\Ut$ such that $\rho (z,\Nt)\le \gamma \rho(z,\Mt\setminus U)$. Choosing $q=\log(2\kappa)$, we get $q(2\kappa)^{1/q}=e\log(2\kappa)$. Hence,
\[
\frac{H(t)}{G(t)}\le e \, \gamma \log(2\kappa)
\begin{cases}
1/ ( (1-t) \rho(x,\Nt) ) & \mbox{ if } 0\le t \le 1/2, \\ 1/( t \rho(y,\Nt)) & \mbox{ if } 1/2\le t \le 1.
\end{cases}
\]

Summing up,
\[
D(x,y)\le 2 e \gamma \log(2\kappa) \rho(x,y) \enpar{\frac{1}{\rho(x,\Nt)} \int_{0}^{1/2} \frac{dt}{1-t} + \frac{1}{\rho(y,\Nt)} \int_{1/2}^1 \frac{dt}{t}},
\]
as desired.
\end{proof}

\begin{lemma}\label{lem:difference}
For $j=1$, $2$, let $a_j$, $b_j\in\RR$ with $a_j\le b_j$. Given $0<p<\infty$ we have
\[
\norm{\chi_{[a_1,b_1)}-\chi_{[a_2,b_2)}}_p\le \enpar{\abs{a_2-a_1}+\abs{b_2-b_1}}^{1/p}.
\]
\end{lemma}

\begin{proof}
Set $\chi_{[a,b)}=-\chi_{[b,a)}$ if $b<a$. By the triangle $p$-law,
\begin{align*}
\norm{\chi_{[a_1,b_1)}-\chi_{[a_2,b_2)}}_p^p
&\le \norm{\chi_{[a_1,b_1)}-\chi_{[a_1,b_2)}}_p^p+\norm{\chi_{[a_1,b_2)}-\chi_{[a_2,b_2)}}_p^p\\
&= \norm{\chi_{[b_2,b_1)}}_p^p+\norm{\chi_{[a_1,a_2)}}_p^p=\abs{b_1-b_2}+\abs{a_2-a_1}.\qedhere
\end{align*}
\end{proof}

The methods we develop here also give quantitative refinements of B\'{\i}ma’s estimates. We explicitly exhibit the functorial character of the construction.

\begin{theorem}\label{thm:AnsoBima}
Let $0<p\le 1$, $\kappa\in\NN$, $\alpha$, $\gamma\in[1,\infty)$, and $\beta\in(0,\infty)$. There is a constant $C=C(p,\kappa,\gamma,\beta,\alpha)$ such that for every closed subspace $\Nt$ of a $p$-metric space $(\Mt,d)$ such that $\Mt\setminus\Nt$ has a Whitney cover with parameters $(\kappa,\gamma,\beta,\alpha)$ relative to the metric space $(\Mt,d^p)$, and any $p$-Banach space $(\XX,\norm{\cdot}_\XX)$ there is a linear map
\[
T\colon \Lip(\Nt,\XX)\to\Lip(\Mt,L_p(\XX))
\]
such that $T(f)|_\Nt= \chi_{[0,1)} \otimes f$ for all $f\in \Lip(\Nt,\XX)$ , and $\norm{T} \le C$. In fact, the result holds with any
\[
C>D:=\enpar{8 \, e \log(2)\, \gamma \, (2+\beta) \frac{(1+\alpha)^2}{\alpha} \kappa \log (2\kappa) }^{1/p}.
\]
\end{theorem}

\begin{proof}
Let $(\varphi_j)_{j\in\Jt}$ be the partition of unity of $V:=\Mt\setminus\Nt$ provided by Lemma~\ref{lem:Bima}. Let also $(x_j)_{j\in\Jt}$, $\nu$, and $\mu$ be as in this lemma. Note that $\mu\ge 2\gamma$. Pick a well-ordering on $\Jt$ and define for each $j\in\Jt$
\[
a_j,b_j \colon V \to [0,1], \quad a_j(x)=\sum_{\substack{k\in\Jt \\ k<j}} \varphi_k(x), \, b_j(x)=\sum_{\substack{k\in\Jt \\ k\le j}} \varphi_k(x).
\]
If $i<j$, then $b_i\le a_j$. For each $x\in V$ there is $j\in\Jt$ such that $b_j(x)=1$. Hence, given $0\le t <1$, the set $\{j\in \Jt \colon t <b_j(x)\}$ is nonempty, whence it has a minimum element, say $j$. Since there is $k<j$ such that $b_k(x)=a_j(x)$, $a_j(x)\le t$. Hence,
\[
I_j(x):=[a_j(x),b_j(x)), \quad j\in\Jt,
\]
is a partition of $[0,1)$. Therefore,
\begin{equation}\label{eq:partition}
\sum_{j\in\Jt} \chi_{I_j(x)}=\chi_{[0,1)}, \quad x\in\Vt.
\end{equation}
Note that $I_j(x)=\emptyset$ unless
\[
j\in \Jt(x):=\enbrace{k\in \Jt \colon \varphi_k(x)>0},
\]
and $\abs{\Jt(x)}\le \kappa$. Set $\Jt(x,y)=\Jt(x)\cup\Jt(y)$ for all $x$, $y\in V$.

We define the wished-for operator $T$ by
\[
T(f)(x)=\begin{cases}
\chi_{[0,1)} \otimes f(x) & \mbox{ if }x\in\Nt,\\
\sum_{j\in\Jt} \chi_{I_j(x)}\otimes f(x_j) & \mbox{ if } x\in V.
\end{cases}
\]
Given $f\in \Lip(\Nt,\XX)$ and $x$, $y\in\Mt$ we set
\[
R(f,x,y)=\norm{T(f)(x)-T(f)(y)}_{p,\XX}.
\]
Suppose that $x\in\Nt$ and $y\in V$. If $j\in\Jt(y)$,
\[
d^p(x,x_j)\le d^p(x,y)+d^p(y,x_j)\le d^p(x,y)+\nu d^p(y,\Nt) \le (1+\nu) d^p(x,y).
\]
Consequently, using \eqref{eq:partition} we obtain
\begin{align*}
R(f,x,y)
&= \norm{\sum_{j\in\Jt} \chi_{I_j(y)}\otimes \enpar{f(x)-f(x_j)}}_{p,\XX}\\
&\le \norm{\sum_{j\in\Jt(y)} \chi_{I_j(y)}\otimes \enpar{f(x)-f(x_j)}}_{p,\XX}\\
&\le\enpar{\sum_{j\in\Jt(y)} \varphi_j (y) \norm{f(x)-f(x_j)}^p}^{1/p}\\
&\le \Lip(f) \enpar{\sum_{j\in\Jt(y)} \varphi_j (y) d^p(x,x_j)}^{1/p}\\
&\le (1+\nu)^{1/p} d(x,y) \Lip(f) \enpar{\sum_{j\in\Jt(y)} \varphi_j (y)}^{1/p}\\
&= C_1^{1/p} \Lip(f) d(x,y),
\end{align*}
where $C_1=1+\nu$.

Suppose that $x$, $y\in V$ and $\Jt(x)\cap\Jt(y)=\emptyset$. For any $z\in\Nt$ we have
\begin{align*}
\enpar{R(f,x,y)}^p
&\le \norm{T(f)(x)-T(f)(z)}_{p,\XX}^p+\norm{T(f)(z)-T(f)(y)}_{p,\XX}^p\\
&\le (1+\nu) \enpar{\Lip(f)}^p \enpar{d^p(x,z)+d^p(y,z)}\\
&\le (1+\nu) \enpar{\Lip(f)}^p \enpar{2d^p(x,z)+d^p(x,y)}.
\end{align*}
Taking the infimum on $z$ we get
\begin{align*}
R(f,x,y)
& \le (1+\nu)^{1/p} \Lip(f) \enpar{2d^p(x,\Nt)+d^p(x,y)}^{1/p}\\
&\le C_2^{1/p} \Lip(f) d(x,y),
\end{align*}
where $C_2=(1+\nu)(1+2\gamma)$.

Finally, suppose that $x$, $y\in\Vt$ and there exists $i\in \Jt(x)\cap\Jt(y)$. By \eqref{eq:partition},
\begin{align*}
R(f,x,y)
&=\norm{\sum_{j\in \Jt} \enpar{\chi_{I_j(x)} -\chi_{I_j(y)}} \otimes \enpar{f(x_j)-f(x_i)}}_{p,\XX}\\
&=\norm{\sum_{j\in \Jt(x,y)} \enpar{\chi_{I_j(x)} -\chi_{I_j(y)}} \otimes \enpar{f(x_j)-f(x_i)}}_{p,\XX}\\
&\le \abs{\Jt(x,y)}^{1/p} M(x,y) N(f,x,y,i),
\end{align*}
where
\begin{align*}
N(f,x,y,i)&=\sup_{j\in\Jt(x,y)} \norm{f(x_j)-f(x_i)}_\XX, \\
M(x,y)&=\sup_{j\in\Jt} \norm{\chi_{I_j(x)} -\chi_{I_j(y)}}_p.
\end{align*}

If $z\in\{x,y\}$ and $j\in\Jt(z)$,
\[
d^p(x_j,x_i) \le d^p(z,x_j)+d^p(z,x_i) \le \nu \enpar{d^p(x,\Nt) + d^p(y,\Nt)}.
\]
Therefore,
\[
\frac{N(f,x,y,i)}{\Lip(f)} \le \sup_{j\in \Jt(x,y)} d(x_j,x_i)
\le \nu^{1/p} \enpar{d^p(x,\Nt) + d^p(y,\Nt)}^{1/p}.
\]

By Lemma~\ref{lem:difference},
\begin{align*}
M(x,y)
&\le \sup_{j\in\Jt} \enpar{\abs{a_j(x)-a_j(y)}+ \abs{b_j(x)-b_j(y)}}^{1/p}\\
&\le 2^{1/p} \enpar{\sum_{j\in\Jt} \abs{\varphi_j(x)- \varphi_j(y)}}^{1/p}\\
&\le 2^{1/p} \mu^{1/p} d(x,y) \enpar{d^{-p} (x,\Nt) + d^{-p} (y,\Nt)}^{1/p}.
\end{align*}
Assume without loss of generality that $d(y,\Nt)\le d(x,\Nt)$. Then there is $s\in[1,\alpha]$ such that
$d^{p} (x,\Nt)=s d^{p} (y,\Nt)$. Since the function $\rho$ defined on $[1,\infty)$ by $\rho(t)=(1+t)(1+t^{-1})$ is icreasing,
\[
\enpar{d^p(x,\Nt) + d^p(y,\Nt)} \enpar{d^{-p} (x,\Nt) + d^{-p} (y,\Nt)}=\rho(s)\le \rho(\alpha).
\]
Finally, $\abs{\Jt(x,y)}\le 2\kappa$.
Summing up,
\[
R(f,x,y)\le C_3^{1/p} \Lip(f) d(x,y),
\]
where $C_3=4 \nu\mu(1+\alpha)(1+\alpha^{-1}) \kappa$. Since
\[
C_1\le C_2\le (1+\nu)(1+\mu) \le 4\nu \mu \le C_3,
\]
and we can choose $\nu$ and $\mu$ so that $C_3$ is arbitrarily close to $D^p$, we are done.
\end{proof}
\section{Canonical embeddings of Lipschitz-free \texorpdfstring{$p$}{}-spaces }\label{sect:CanEmb}\noindent
Let $0<p\le 1$. Given a pointed $p$-metric space $\Mt$ there is a unique $p$-Banach space $\F_p(\Mt)$ such that $\Mt$ embeds isometrically in $\F_p(\Mt)$ via a canonical map denoted $\delta_{p,\Mt}$, and, given an arbitrary $p$-Banach space $\XX$, every Lipschitz map $f\colon\Mt\to\XX$ with Lipschitz constant $\Lip(f)$ that maps the base point $0$ in $\Mt$ to the null vector in $\XX$ extends to a unique linear bounded map $E[f;p]\colon \F_p(\Mt)\to\XX$ with $\Vert E[f;p]\Vert =\Lip(f)$. Pictorially,
\[
\xymatrix{\Mt \ar[r]^{f} \ar[d]_{\delta_{p,\Mt}} & \XX \\
\F_p(\Mt) \ar[ru]_{E[f;p]} & }
\]

These spaces, called Lipschitz-free $p$-spaces, provide a canonical linearization process at level $p$ of Lipschitz maps between $p$-metric spaces. Indeed, given $0<p\le 1$, $p$-metric spaces $\Mt$, and $\Nt$ and a Lipschitz map $h\colon\Nt\to\Mt$ with $h(0)=0$, the linear bounded map
\[
L[h;p]:=E[\delta_{p,\Mt} \circ h;p]
\]
makes the diagram
\[
\xymatrix{\Nt \ar[r]^{h} \ar[d]_{\delta_{p,\Nt}} &\Mt \ar[d]^{\delta_{p,\Mt}} \\
\F_p(\Nt) \ar[r]_{L[h;p]} &\F_p(\Mt)}
\]
commute. Note that $\norm{L[h;p]}=\Lip(h)$. Given a pointed $p$-metric space $\Mt$ and $0\in\Nt \subset\Mt$, we denote by
\[
L[\Nt,\Mt;p]\colon\F_p(\Nt) \to \F_p(\Mt)
\]
the canonical linearization at level $p$ of the inclusion map of $\Nt$ into $\Mt$.

We will make use of a concrete implementation of Lipschitz-free $p$-spaces. Let $\enpar{\delta_\Mt(x)}_{x\in\Mt\setminus\{0\}}$ be the canonical basis of the free vector space $\Pt(\Mt)$ over $\Mt\setminus\{0\}$, and put $\delta_\Mt(0)=0$. Vectors in $\Pt(\Mt)$ are usually called \emph{molecules}. Given $\mu=\sum_{x\in \Mt} a_x\delta_\Mt(x)\in\Pt(\Mt)$, we set
\begin{equation*}
\norm{\mu}_{\F_p(\Mt)}=\sup \norm{\sum_{x\in\Mt} a_x f(x) }_\XX
\end{equation*}
the supremum being taken over all $p$-normed spaces $\enpar{\XX,\norm{\cdot}_{\XX}}$ and all $1$-Lipschitz maps $f\colon \Mt\to\XX$ with $f(0)=0$. Solving a question posed in \cite{AlbiacKalton2009}, the authors of \cite{AACD2020} proved $\norm{\mu}_{\F_p(\Mt)}>0$ unless $\mu=0$.
\begin{theorem}[\cite{AACD2020}*{Proposition 4.1 and Theorem 4.10}]\label{thm:AACD}
Let $0<p\le 1$ and $(\Mt,d)$ be a $p$-metric space. Then $\norm{\cdot}_{\F_p(\Mt)}$ is a $p$-norm over $\Pt(\Mt)$. Besides, for all $\mu\in\Pt(\Mt)$,
\[
\norm{\mu}_{\F_p(\Mt)}=\inf \enpar{\sum_{j\in F} \abs{a_j}^p}^{1/p},
\]
the infimum being taken over all finite families $(a_j)_{j\in F}$ in $\RR$ for which there are $(x_j)_{j\in F}$ and $(y_j)_{j\in F}$ in $\Mt$ such that $x_j\not=y_j$ for all $j\in F$, and
\begin{equation}\label{eq:expansion}
\mu=\sum_{j\in F} a_j \frac{\delta_\Mt(x_j) - \delta_\Mt(y_j) }{d(x_j,y_j)}.
\end{equation}
\end{theorem}

The Lipschitz-free $p$-space $\F_p(\Mt)$ is just the completion of the $p$-normed space $(\Pt(\Mt), \norm{\cdot}_{\F_p(\Mt)})$, and the canonical embedding of $\Mt$ into $\F_p(\Mt)$ is just $\delta_{p,\Mt}$ regarded as a map into $\F_p(\Mt)\supseteq \Pt(\Mt)$. If $\Mt$ is a quasi-metric space, $\Nt$ is a subspace of $\Mt$, and $\mu\in\Pt(\Mt)$ admits an expansion as in \eqref{eq:expansion} with $x_j$, $y_j\in \Nt$ for all $j\in F$, we say that $\mu$ is \emph{supported} on $\Nt$. More generally, we say that a vector in $\F_p(\Mt)$ is supported on $\Nt$ if it is a limit point of molecules supported on $\Nt$.

If $p=1$, $\F_p(\Mt)$ is the classical Lipschitz-free space (also known as Arens--Eells space or transportation cost space, depending on background or inclinations) $\F(\Mt)$ widely studied from several angles (see e.g. \cites{WeaverBook2018,GodefroyKalton2003,OO19}).

The following result from \cite{CuthRaunig2024} allows us to restrict ourselves to Lipschitz free spaces over trees. Recall that a \emph{tree} is a connected graph $G$ that contains no nontrivial cycle. A positive weight $\ww$ on the edges of $G$ induces a metric $\rho_\ww$ on the one-dimensional simplicial complex $\Sigma(G)$ associated with $G$. Given $0<p\le 1$, we say that $(\Sigma(G),\rho_\ww^{1/p})$ is the \emph{$p$-metric tree associated with the weighted tree $(G,\ww)$}. We call its restriction to the set of vertices of $G$, say $V$, the \emph{skeleton $p$-metric tree associated with $(G,\ww)$}. Choose a base point $0\in V$. We say that $v\in V$ is a \emph{leaf} of the rooted tree $(G,0)$ if the unique path $(v_j)_{j=0}^n$ from $0$ to $v$ consisting of distint vertices of $G$ can not be extended. We call $\leaf(G,0)$ the set consisting of the base point and all leaves. By extension, if $\Mt$ is the pointed skeleton $p$-metric tree associated with a weighted rooted tree $(G,\ww,0)$, we set $\leaf(\Mt)=\leaf(G,0)$.

\begin{lemma}[\cite{CuthRaunig2024}*{Theorem 4.16}]\label{lem:CR2024}
Let $0<p\le q\le 1$. Suppose that there is a constant $C(p,q)$ such that $L:=L[\leaf(\Mt),\Mt;p]$ is an isomorphic embedding with $\norm{L^{-1}} \le C(p,q)$ for every pointed skeleton $q$-metric tree $\Mt$. Then, $L:=L[\Nt,\Mt;p]$ is an isomorphic embedding with $\norm{L^{-1}} \le C(p,q)$ for every pointed $q$-metric space $\Mt$ and every $0\in\Nt \subset\Mt$.
\end{lemma}

Lemma~\ref{lem:CR2024} makes the following two results even more relevant. A metric space $(\Mt,\rho)$ is said to have the \emph{Nagata $(n,\lambda)$-property}, $\lambda\in(0,\infty)$, $d\in\NN_0$, if for every $s>0$, there exists a family $\Kt$ of nonnempty subsets of $\Mt$ such that
\begin{enumerate}[label=(\alph*),leftmargin=*]
\item $\Kt$ covers $\Mt$, i.e., $\cup_{K\in\Kt} K=\Mt$
\item $\diam(K)\le \lambda s$ for all $K\in\Kt$, and
\item $\abs{\enbrace{K\in\Kt \colon K\cap A\not=\emptyset}}\le n+1$ for every $A\subset\Nt$ with $\diam(A)\le s$.
\end{enumerate}
The minimum $n\in\NN_0$ such that $\Mt$ has the Nagata $(n,\lambda)$-property for some $\lambda\ge 1$ is called the \emph{Nagata dimension of $\Mt$.} Note that the Nagata property and the Nagata dimension plainly pass to subspaces.

\begin{theorem}[\cite{LangSch2005}*{Proposition 3.2}]\label{thm:LS}
Any metric tree has the Nagata $(1,6)$-property. In particular, it has Nagata dimension at most one.
\end{theorem}

\begin{theorem}[cf.\@ \cite{Bima2025}*{Proposition 13}]\label{thm:NagataWhitney}
Let $\Nt$ be a closed subset of a metric space $(\Mt,\rho)$. Assume that $\Nt$ has the Nagata $(n,\lambda)$-property for some $n\in\NN_0$ and $\lambda\ge 1$. Given $R\in(1,\infty)$, there exists a Whitney cover of $\Mt\setminus \Nt$ relative to $\Mt$ with parameters
\[
\enpar{3(n + 1), \frac{R^2}{R-1} , 2(R^2+R-1)(\lambda+1), R^2+R-1}.
\]
\end{theorem}

\begin{proof}
For each $s\in(0,\infty)$, let $\Kt[s]$ be the cover of $\Nt$ whitnessing that it has the Nagata $(d,\lambda)$-property. Given
$A\subset \Nt$ and $t\in(0,\infty)$ we set
\begin{align*}
\Kt[A,s]&=\enbrace{K\in\Kt[s] \colon K\cap A\not=\emptyset}, \\
D[A]&=\enbrace{x\in \Mt\setminus \Nt \colon \rho(x,A)=\rho(x,\Nt)}.
\end{align*}
Note that, by assumption, $\abs{\Kt[A,s]}\le n+1$ provided that $\diam(A)\le s$. Our proof relies on the following facts.

\begin{claim}\label{claim:a}
If $x\in\Mt\setminus\Nt$ and $s\in(0,\infty)$ are such that $2 \rho(x,\Nt)<s$, then $x\in D[K]$ for some $K\in\Kt[s]$.
\end{claim}

\begin{claim}\label{claim:b}
Given $s\in(0,\infty)$ and $x\in\Mt\setminus\Nt$,
\[
\abs{\enbrace{K\in\Kt[s] \colon \rho(x,K)< s/2}}\le n+1.
\]
\end{claim}

The proofs of both claims rely on the fact that $\diam(B(x,s/2))\le s$. To see the former, we note that $\Kt_0:=\Kt\enbrak{B\enpar{x,s/2}\cap\Nt,s}$ is finite, and
\[
\rho(x,\Nt)=\rho\enpar{x,\Nt\cap B\enpar{x,\frac{s}{2}}}= \inf \enbrace{\rho(x,K) \colon K \in \Kt_0}.
\]
To see the latter, we note that if $\rho(x,K)<s/2$, then $K$ intersects $B(x,s/2)$.

We return to the main argumental line. Pick $R\in(1,\infty)$ and set
\[
s_j= 2\enpar{R^2+R-1} R^{j-1}, \quad j\in\ZZ.
\]
Consider the set of indices
\[
\It=\enbrace{(j,K) \colon j\in\ZZ, \, K\in\Kt[s_j]}.
\]
For each $(j,K)\in\It$ we put
\[
A_{j,K}=\At\enpar{\Nt;[R^j,R^{j+1})} \cap D[K], \quad U_{j,K}=\At\enpar{A_{j,K};[0, (R-1)R^{j-1})}.
\]
The set $U_{j,K}$ is open,
\[
U_{j,K}\subset \At\enpar{\Nt; (R^{j-1},(R^2+R-1)R^{j-1})},
\]
and
\begin{align*}
\diam(U_{j,K}) & \le 2 (R-1) R^{j-1} + \diam(A_{j,K}) \\
&\le 2 (R-1) R^{j-1} + 2 R^{j+1} +\diam(K) \\
&\le 2 (R^2+R-1) R^{j-1} + \lambda s_j.
\end{align*}
Consequently, $\rho(x,\Nt) < (R^2+R-1) \rho(y,\Nt)$ for all $x$, $y\in U_{j,K}$, and
\[
\frac{\diam(U_{j,K})}{\rho(U_{j,k},\Nt)}\le 2(R^2+R-1) + 2 (R^2+R-1)\lambda.
\]

Given $x\in\Mt\setminus\Nt$, there is $j\in\ZZ$ such that $\rho(x,\Nt)\in[R^j, R^{j+1})$. Since $2R^{j+1}\le s_j$, by Claim~\ref{claim:a} there is $K\in\Kt[s_j]$ such that $x\in A_{j,K}$. Therefore, $\rho(x,\Mt \setminus U_{j,K})\ge (R-1)R^{j-1}$. Consequently,
\[
\rho(x,\Nt) < \frac{R^2}{R-1} \rho(x,\Mt\setminus U_{j,K}).
\]

Pick $y\in\Mt\setminus\Nt$. Let $i\in\ZZ$ be such that $d(y,\Nt)\in [R^{i},R^{i+1})$. Set
\[
\It[y]=\enbrace{(j,K) \in \It \colon y \in U_{j,K}}.
\]
If $(j,K)\in\It[y]$, then, the intervals $ [R^{i},R^{j+1})$ and $(R^{j-1},R^{j+2})$ overlap. Consequently, $j\in \{i-1,i,i+1\}$. Besides, since there is $x\in A_{j,K}$ such that $\rho(x,y)<(R-1)R^{j-1}$,
\[
\rho(y,K)<(R-1)R^{j-1} + \rho(x,\Nt)< (R-1)R^{j-1}+R^{j+1}=\frac{s_j}{2}.
\]
We infer from Claim~\ref{claim:b} that $\abs{\It[y]}\le 3(n+1)$.
\end{proof}

We are now ready to prove the main result of this section.

\begin{theorem}\label{thm:LIEP}
Given $0<p\le 1$, there is a constant$A(p)$ so that $L:=L[\Nt,\Mt;p]$ is an isomorphic embedding with $\norm{L^{-1}} \le C(p)$ for every pointed a $p$-metric space $(\Mt,d)$ and every $0\in\Nt \subset\Mt$. In fact, we can choose $C(p)=\EC^{1/p}$, where
\[
\EC=16 \, e \log(2) \, \log (12) \min_{R>1}\frac{R^4 (R+1)^2 (7 R^2+7R-6) }{(R-1)(R^2+R-1)}.
\]
\end{theorem}

\begin{proof}
By Lemma~\ref{lem:CR2024}, Theorem~\ref{thm:LS}, and Theorem~\ref{thm:NagataWhitney}, we can assume that $\Mt\setminus\Nt$ has for every $R\in(1,\infty)$ a Whitney cover relative to $(\Mt,d^p)$ with parameters $(\kappa,\gamma,\beta,\alpha)$, where
\[
\kappa=6, \quad \alpha=R^2+R-1, \quad \beta= 14(R^2+R-1), \quad \gamma=\frac{R^2}{R-1}.
\]
Since
\begin{multline*}
8 \, e \log(2) \gamma (2+\beta) \frac{(1+\alpha)^2}{\alpha} \kappa \log (2\kappa) \\
=A(R):=16 \, e \log(2) \, \log (12) \, \frac{R^4 (R+1)^2 (7 R^2+7R-6) }{(R-1)(R^2+R-1)},
\end{multline*}
Theorem~\ref{thm:AnsoBima} applies with any $C>(A(R))^{1/p}$. Let $T$ be the operator with $\norm{T}\le C$ provided by this extension theorem.

Let $(a_x)_{x\in\Nt}$ be an eventually null family of scalars, and set
\[
\mu=\sum_{x\in\Nt} a_x \delta_{\Nt}(x).
\]
Let $(\XX,\norm{\cdot}_\XX)$ be a $p$-Banach space and $f\colon\Nt\to \XX$ be a Lipschitz map with $\Lip(f)\le 1$. Since $\Lip(T(f))\le C$,
\begin{multline*}
C \norm{L(\mu)}_{\F_p(\Mt)}
\ge \norm{\sum_{x\in\Nt} a_x T(f)(x)}_{p,\XX}
= \norm{\sum_{x\in\Nt} a_x \chi_{[0,1)} \otimes f(x)}_{p,\XX}\\
=\norm{\chi_{[0,1)}}_p \norm{\sum_{x\in\Nt} a_x f(x)}_\XX
= \norm{\sum_{x\in\Nt} a_x f(x)}_\XX.
\end{multline*}
Taking the supremum over $f$ we get $\norm{\mu}_{\F_p(\Nt)}\le C \norm{L(\mu)}_{\F_p(\Mt)}$. Taking the infimum on $C\in(A(R))^{1/p},\infty)$ and $R\in(1,\infty)$, we get $\norm{L^{-1}}\le \EC^{1/p}$.
\end{proof}

\begin{remark}
The universal constant $\EC$ in Theorem~\ref{thm:LIEP} has the decimal aproximation
\[
\EC \approx 13982.5641659317.
\]
\end{remark}
\section{Faithfulness of the scale of Lipschitz-free spaces over quasi-metric spaces}\label{sect:Envelope}\noindent
Given a quasi-Banach space $\XX$ and $0<q\le 1$, a \emph{$q$-Banach envelope of $\XX$} is a pair $(\Env[q]{\XX},J[q,\XX])$ satisfying the following properties.
\begin{enumerate}[label=(E.\arabic*), leftmargin=*]
\item $\Env[q]{\XX}$ is a $q$-Banach space.
\item The canonical map $J[q,\XX]\colon \XX\to \Env[q]{\XX}$ is a linear contraction.
\item\label{it:UPEnv} For every $q$-Banach space $\YY$ and every linear contraction $T\colon \XX\to\YY$ there is a unique linear contraction $\Env[q]{T}\colon \Env[q]{\XX} \to \YY$ such that $\Env[q]{T} \circ J[q,\XX]=T$. Pictorially,
\[
\xymatrix{\XX \ar[r]^{T} \ar[d]_{J[q,\XX]} & \YY \\
\Env[q]{\XX} \ar[ru]_{\Env[q]{T} } & }
\]
\end{enumerate}
Since \ref{it:UPEnv} is a universal property, the $q$-Banach envelope of $\XX$ is essentially unique. The $1$-Banach envelope of a quasi-Banach space $\XX$ is just the standard Banach envelope, which has been widely studied since its introduction in the 1970’s by Peetre \cite{Peetre1974} and Shapiro \cite{Shapiro1977}.

To place Question~\ref{qt:AKEnv} in context, recall that for $0<p<1$ the family of $q$-Banach envelopes $\Env[q]{\XX}$ of a $p$-Banach space $\XX$ with $p\le q\le 1$ forms a natural convexification hierarchy interpolating between the original quasi-Banach structure and the Banach envelope. Indeed, given $p\le q \le r$, $( \Env[r]{\XX},\Env[q]{J[r,\XX]})$ is the $r$-Banach envelope of $\Env[q]{\XX}$, and the diagram
\begin{equation}\label{eq:hierarchy}
\begin{gathered}
\xymatrix{\XX \ar[rr]^{J[r,\XX]} \ar[dr]_{J[q,\XX]} && \Env[r]{\XX} \\
& \Env[q]{\XX} \ar[ru]_{\Env[q]{J[r,\XX]}} & }
\end{gathered}
\end{equation}
commutes. In many settings, this intermediate convexification captures geometric features that are not visible at the endpoint $q=1$ but already lies beyond the reach of tools available at level $p$.

Let $0<p\le q\le 1$ and $\Mt$ be a pointed $q$-metric space. Consider the canonical map
\[
J_{p,q}[\Mt]:=E[ \delta_{q,\Mt};p]\colon \F_p(\Mt) \to \F_q(\Mt).
\]
It is known \cite{AACD2020}*{Proposition 4.20} that the $q$-Banach envelope of $\F_p(\Mt)$ is $(\F_q(\Mt),J_{p,q}[\Mt])$, and this identification does not always extend to $r$-envelopes for $r>q$. So, in this setting the envelope hierarchy
takes the form
\begin{equation}\label{eq:LFhierarchy}
\begin{gathered}
\xymatrix{\F_p(\Mt) \ar[rr]^{J_{p,q}[\Mt]} \ar[dr]_{J_{p,r}[\Mt]} && \F_q(\Mt) \\
&\F_r(\Mt) \ar[ru]_{J_{r,q}[\Mt] } & },
\end{gathered}
\qquad
\begin{gathered}
p\le r \le q.
\end{gathered}
\end{equation}
Unless $\Mt$ is a metric space, in which case $q=1$, this scale of quasi-Banach spaces stops at $q<1$.

Te injectivity of $J_{p,q}[\Mt]$ expresses that no nonzero vector of $\F_{p}(\Mt)$ becomes trivial in $\F_{q}(\Mt)$. In other words, the $q$-Banach envelope retains the full information encoded at level $p$. By Theorem~\ref{thm:AACD}, the envelope map $J_{p,q}[\Mt]$ is injective on a dense subspace of $\F_p(\Mt)$, but this does not guarantee that $J_{p,q}[\Mt]$ is one-to-one (see \cite{AAW2021c}). While a recent result from \cite{AlbiacAnsorena2026b} shows that such degeneration does not occur when the diagram \eqref{eq:LFhierarchy} reaches the optimal index $q=1$, the behavior of this hierarchy has remained unknown for the other values of $p<q<1$. To fill this gap in the theory, we will use a range of techniques that we detail below.
\subsection{\texorpdfstring{$\SL_\infty$}{}-spaces}
Following \cites{LinPel1968,LinRos1969}, we say that a Banach space $\XX$ is a \emph{$\SL_\infty$-space} if there is a constant $C$ such that for every finite-dimensional subspace $\VV\subset\XX$ there are $d\in\NN$ and linear maps $P\colon\XX\to \ell_\infty^d$ and $J\colon \ell_\infty^d \to \XX$ such that $\norm{J} \norm{P} \le C$, $P\circ J=\Id_{\ell_\infty^d}$, and $\VV\subset J(\ell_\infty^d)$. If $D\in[1,\infty)$ is such that we can choose any $C>D$, we say that $\XX$ is an $\SL_{\infty,D}$-space.

\begin{theorem}\label{thm:metricembedding}
For any $0<q\le 1$ and any $q$-metric space $\Mt$ there is an $\SL_{\infty,1}$-space $(\XX,\norm{\cdot}_\XX)$ such that $\Mt$ isometrically embeds into the $q$-metric space $(\XX, \norm{\cdot}^{1/q}_\XX)$.
\end{theorem}

\begin{proof}
Since $(\Mt,d^q)$ is a metric space, it isometrically embeds into a $\SL_{\infty,1}$-space $(\XX,\norm{\cdot}_\XX)$ (see \cite{AlbiacAnsorena2026b}*{Theorem 2.5}). This means that $(\Mt,d)$ isometrically embeds into $(\XX, \norm{\cdot}^{1/q}_\XX)$.
\end{proof}
\subsection{Finite-dimensional decompositions}
A \emph{Schauder decomposition} of a quasi-Banach space $\XX$ is a family $\XB:=(\XX_n)_{n=1}^\infty$ of closed subspaces of $\XX$ such that for every $x\in\XX$ there is a unique sequence $(x_n)_{n=1}^\infty$ such that $x_n\in\XX_n$ for all $n\in\NN$, and $x=\sum_{n=1}^\infty x_n$. We call the linear bounded mapping $x\mapsto x_n$ the $n$th coordinate projection relative to $\XB$. If $\dim(\XX_n)<\infty$ for all $n\in\NN$, we say that $\XB$ is a \emph{finite-dimensional Schauder decomposition} of $\XX$. A quasi-Banach has a finite-dimensional Schauder decomposition if and only if there is a sequence $(S_n)_{n=1}^\infty$ of endomorphisms of $\XX$ such that $S_n\circ S_m=S_{\min\{n,m\}}$ for all $m$, $n\in\NN$, $\YY_n:=S_n(\XX)$ is finite-dimensional for all $n\in\NN$, $\cup_{n=1}^\infty \YY_n$ is dense in $\XX$, and $\sup_n \norm{S_n}<\infty$. In fact, if $(S_n)_{n=1}^\infty$ is as above, then
\[
\YY_n\cap \Ker(S_{n-1}), \quad n\in\NN,
\]
is a finite-dimensional Schauder decomposition.

A \emph{Markusevich basis} of a quasi-Banach space $\XX$ is a family $\XB=(\xx_\gamma)_{\gamma\in\Gamma}$ in $\XX$ for which there is $\XB^*=(\xx_\gamma^*)_{\gamma\in\Gamma}$ in $\XX^*$ such that
\begin{itemize}
\item $\xx_\gamma^*(\xx_\alpha)=\delta_{\alpha,\gamma}$ for all $\alpha$, $\gamma\in\Gamma$,
\item the linear span of $\XB$ is dense in $\XX$, and
\item the linear span of $\XB^*$ is weak$^*$-dense in $\XX^*$, that is, the coefficient transform defined on $\XX$ by
\[
f\mapsto \enpar{\xx_\gamma^*(f)}_{\gamma\in\Gamma},
\]
is one-to-one.
\end{itemize}

\begin{lemma}\label{lem:SDImpliesMar}
Let $\XX$ be a quasi-Banach space with a finite-dimensional Schauder decomposition. Then, $\XX$ has a Markusevich basis.
\end{lemma}

\begin{proof}
For each $n\in\NN$, let $T_n\colon\XX\to\XX_n$ be the $n$th coordinate projection relative to a finite-dimensional Schauder decomposition $(\XX_n)_{n=1}^\infty$. Let $(x_{j,n})_{j\in \Nt_n}$ be an algebraic basis of $\XX_n$ with dual basis $(x_{j,n}^*)_{j\in \Nt_n}$. If
\[
\Nt=\enbrace{(j,n) \colon n\in\NN, \, j\in\Nt_n}
\]
is the disjoint union of $(\Nt_n)_{n=1}^\infty$, then $(x_{j,n})_{(j,n)\in\Nt}$ is Markusevich basis of $\XX$ with coefficient functionals $(x_{j,n}^* \circ T_n)_{(j,n)\in\Nt}$.
\end{proof}

When studying nonlocally convex quasi-Banach spaces, we must do without the tools that rely on the Hahn-Banach theorem (such as duality techniques) or Bochner's integration. This important initial drawback has led functional analysts to focus primarily on locally convex spaces since the dawn of the theory while nonlocally convex topological vector spaces were pushed to the back burner for the first half of the 20th century. In spite of that, the theory of quasi-Banach spaces was gradually developed, and in that process the Banach envelope played a crucial role. It is known (see e.g.\@ \cite{AABW2021}*{\S 9}) that the envelope map of a quasi-Banach space $\XX$ is one-to-one if and only if its dual space $\XX^*$ separates the points of $\XX$. Hence, the envelope map of a quasi-Banach space with a Markusevich basis is injective. Combining this observation with the commutative diagram \eqref{eq:hierarchy} yields the following consequence of Lemma~\ref{lem:SDImpliesMar}.

\begin{proposition}\label{prop:SchauderSeparates}
Let $0<q\le 1$ and $\XX$ be a quasi-Banach space with a finite-dimensional Schauder decomposition. Then the $q$-envelope map of $\XX$ is one-to-one.
\end{proposition}

Proposition~\ref{prop:SchauderSeparates} will be useful for our purposes thanks to the following result.

\begin{theorem}\label{thm:FDDZdq}
Let $0<p\le 1$, $0<p \le q<\infty$ and $d\in\NN$. Then $\F_p(\ZZ^d,\abs{\cdot}_\infty^{1/q})$ has a a finite-dimensional Schauder decomposition.
\end{theorem}

\begin{proof}
Given $n\in\NN$, the mapping
\[
r_n\colon \ZZ^d \to \ZZ^d, \quad w \mapsto \min\enbrace{1, \frac{n}{\abs{w}_\infty}} w
\]
is a $1$-Lipschitz retraction onto the finite set
\[
\enbrace{w\in\ZZ^d \colon \abs{w}_\infty\le n}.
\]
The linearizations $(L[r_n;p])_{n=1}^\infty$ witness the existence of a Schauder decomposition.
\end{proof}
\subsection{Lipschitz-free spaces over \texorpdfstring{$\RR^d$}{} vs.\ Lipschitz-free spaces over \texorpdfstring{$\ZZ^d$}{}, \texorpdfstring{$d\in\NN$}{}}
The following result concerning the geometry of Lipschitz-free spaces over anti-snowflakings of (whole) Euclidean spaces can fairly be considered the crucial fact that will allow us to answer Question~\ref{qt:AKEnv}. It will be convenient to consider several norms on $\RR^d$, $d\in\NN$.

Given $0<q\le \infty$, a measure space space $(\Omega,\Sigma,\mu)$, quasi-Banach spaces $\XX$ and $\YY$, and an operator $T\colon \XX \to \YY$,
\[
V_{q,\mu}[T]\colon L_q(\mu,\XX)\to L_q(\mu,\YY)
\]
stands for the operator defined by $V_{q,\mu}[T](f)(x)= T(f(x))$ $\mu$-a.e.\@ $x\in\Omega$ for all $f\in L_q(\mu,\XX)$. If $\mu$ is the Lebesgue measure on an Euclidean space $[0,1)^d$, we set $V_{q,\mu}[T]=V_{q,d}[T]$.

\begin{theorem}\label{thm:Embedding}
Let $0<p\le q\le 1$ and $d\in\NN$. There is a constant $C=C(p,q,d)$ so that for each $t \in(0,\infty)$ there is a bounded linear operator
\[
T_{t,p}=T_t[p,q,d]\colon\F_p\enpar{\RR^d,\abs{\cdot}_1^{1/q}}\to L_q\enpar{[0,1)^d,\F_p\enpar{t \ZZ^d,\abs{\cdot}_\infty^{1/q}}}
\]
such that $T_{t,p}(L_{t,p}(\mu))= \chi_{[0,1)^d}\otimes \mu$ for all $\mu\in \F_p(t \ZZ^d,\abs{\cdot}^{1/q})$, where
\[
L_{t,p}= L\enbrak{\enpar{t \ZZ^d, \abs{\cdot}_\infty^{1/q}}, \enpar{\RR^d,\abs{\cdot}_\infty^{1/q}};p}
\]
is the canonical linearization at level $p$ of the embedding of $t \ZZ^d$ into $\RR^d$, and $ \norm{T_{t,p}}\le C$. In fact, we can choose $C=2^{1/p-1/q}(2^{2d}-1)^{1/p}$. Furthermore, given $p\le r \le q$, the diagram
\[
\xymatrixcolsep{8pc}
\xymatrix{\F_p\enpar{\RR^d,\abs{\cdot}_1^{1/q}} \ar[r]^{T_t[p,q,d]} \ar[d]_{J_{p,r}\enbrak{\RR^d,\abs{\cdot}^{1/q}}}&L_q\enpar{[0,1)^d,\F_p\enpar{t\ZZ^d,\abs{\cdot}_\infty^{1/q}}} \ar[d]^{V_{q,d}\enbrak{J_{p,r}\enbrak{t\ZZ^d,\abs{\cdot}^{1/q}}}} \\
\F_r\enpar{\RR^d,\abs{\cdot}_1^{1/q}} \ar[r]_{T_t[r,q,d]}&L_q\enpar{[0,1)^d,\F_r\enpar{t\ZZ^d,\abs{\cdot}_\infty^{1/q}}}}
\]
commutes.
\end{theorem}

To prove Theorem~\ref{thm:Embedding} we shall develop some ideas that originally emerged from \cite{AACD2022}. Let $\sigma\colon\RR \to\RR$ be the symmetry given by $x\mapsto 1-x$. For $x\in\RR$ and $w\in\ZZ$ we define
\[
R[w,x]
=\begin{cases}
[0,1+w-x) & \mbox{ if } -1\le w-x \le 0, \\ [w-x,1) & \mbox{ if } 0\le w-x \le 1, \\ \emptyset & \mbox{ if } \abs{w-x} \ge 1.
\end{cases}
\]
If $x=w$, the two possible ways of defining $R[w,x]$ yield the whole set $[0,1)$, while if $x-w\in\{-1,1\}$, the two possible ways of defining it yield the empty set. So, $R$ is well-defined. Moreover,
\begin{enumerate}[label=(A.\arabic*)]
\item\label{it:1:OD} $\enbrace{w\in\ZZ \colon R[w,x]\neq \emptyset}=\enbrace{\enfloor{x}, \enceil{x}}$ for all $x\in\RR$;
\item\label{it:2:OD} $(R[w,x])_{w\in\ZZ}$ is a partition on $[0,1)$ for all $x\in\RR$;
\item\label{it:3:OD} $R[w-u,x-u]=R[w,x]$ for all $u$, $w\in\ZZ$ and $x\in\RR$;
\item\label{it:4:OD} $\sigma(R[w,x]) \triangle R[-w,-x]$ is a null set for all $w\in\ZZ$ and $x\in\RR$;
\item\label{it:5:OD} $\abs{R[0,x]-R[0,y]}=\abs{x-y}$ for all $x$, $y\in[0,1]$.
\end{enumerate}
From \ref{it:3:OD}, \ref{it:4:OD}, \ref{it:5:OD}, and the triangle inequality, we infer that
\begin{enumerate}[label=(A.\arabic*),resume]
\item\label{it:7:OD} $\abs{R[w,x]-R[w,y]}\le \abs{x-y}$ for all $w\in\ZZ$ and $x$, $y\in\RR$.
\end{enumerate}

Given $d\in\NN$, let $\Rt_d$ be the set of all left-closed right-open intervals contained in $[0,1)^d$. Given $w=(w_j)_{j=1}^d\in\ZZ^d$, and $x=(x_j)_{j=1}^d\in\RR^d$ we set
\[
V_d(x):=\prod_{j=1}^d \enbrace{\enfloor{x_j}, \enceil{x_j}},
\]
we define $R_d[w,x]\in\Rt_d$ by
\[
R_d[w,x]=\prod_{j=1}^d R[w_j,x_j],
\]
and we define $\varphi_d[w,x]\colon[0,1)^d \to \{0,1\}$ as the indicator function of $R_d[w,x]$. The function $R_d$ inherits properties \ref{it:1:OD}, \ref{it:2:OD}, \ref{it:3:OD}, \ref{it:4:OD} and \ref{it:7:OD} from $R$. Let us specify those properties that are of interest to us.
\begin{enumerate}[label=(B.\arabic*),leftmargin=*]
\item\label{it:1:DD} $\enbrace{w\in\ZZ^d \colon R[w,x]\neq \emptyset}=V_d(x)$ for all $x\in\RR^d$.
\item\label{it:2:DD} $(R_d[w,x])_{w\in\ZZ^d}$ is a partition of $[0,1)^d$ for all $x\in\RR^d$.
\item\label{it:21:DD} $R_d(w,w)=[0,1)^d$ for all $w\in\ZZ^d$.
\item\label{it:3:DD} $R_d[w-u,x-u]=R_d[w,x]$ for all $u$, $w\in\ZZ^d$ and $x\in\RR^d$, and
\item\label{it:6:DD} If $x\in\RR^d$, and $w\in\RR^d$, then $\abs{R_d[w,x]}\le 1$.
\item\label{it:7:DD} If $x$, $y\in\RR^d$, and $w\in\RR^d$, then $\abs{R_d[w,x]-R_d[w,y]}\le\abs{x-y}_1$.
\end{enumerate}

\begin{proof}[Proof of Theorem~\ref{thm:Embedding}]
We choose the null vector of $\RR^d$ as the base point of $\RR^d$ and $t\ZZ^d$. We claim that it suffices to prove the result for $t=1$. To that end, we consider for each $t>0$ the dilation $\delta_t\colon \RR^d \to \RR^d$ given by $x\mapsto t x$, and its restrictions $\gamma_t\colon t^{-1} \ZZ^d \to \ZZ^d$ and $\beta_t\colon \ZZ^d \to t \ZZ^d$. Making the chosen level $p$ explicit we set,
\[
\XX_p:=\F_p\enpar{\RR^d,\abs{\cdot}_1^{1/q}}, \,
\YY_{t,p}= \F_p\enpar{t \ZZ^d,\abs{\cdot}_\infty^{1/q} },\,
\LL_{t,p}= L_q \enpar{[0,1)^d ,\YY_t}.
\]
Put $\YY_p=\YY_{1,p}$, $\LL_p=\LL_{1,p}$ and $L_p=L_{1,p}$. Let
\[
D_{t,p}\colon\XX\to\XX, \quad C_{t,p}\colon \YY_{1/t,p} \to \YY_p, \quad B_{t,p}\colon \YY_p \to \YY_{t,p}
\]
be the linearizations at level $p$ of $\delta_t$, $\gamma_t$ and $\beta_t$ respectively. Finally, define $A_{t,p}\colon\LL_p \to\LL_{t,p}$ by $A_{t,p}=V_{q,d}[B_{t,p}]$. We have
\[
D_{t,p} \circ L_{1/t,p}=L_p\circ C_{t,p}, \quad
B_{t,p}\circ C_{1/t,p}=\Id_{\YY_{t,p}},
\]
and $\norm{A_{t,p}}=\norm{B_{t,p}}=\norm{C_{t,p}}=\norm{D_{t,p}}=t^{1/q}$. Furthermore, the diagrams
\begin{equation}\label{eq:diagX}
\begin{gathered}
\xymatrixcolsep{6pc}
\xymatrix{\XX_p \ar[d]_{J_{p,r}\enbrak{\RR^d,\abs{\cdot}^{1/q}}} \ar[r]^{D_{t,p}} & \XX_{p}
\ar[d]^{J_{p,r}\enbrak{\RR^d,\abs{\cdot}^{1/q}}} \\
\XX_r \ar[r]_{D_{t,r}}& \XX_{r},}
\end{gathered}
\end{equation}
\begin{equation}\label{eq:diagL}
\begin{gathered}
\xymatrixcolsep{6pc}
\xymatrix{\LL_p \ar[d]_{V_{q,d}\enbrak{J_{p,r}\enbrak{\ZZ^d,\abs{\cdot}^{1/q}}}} \ar[r]^{A_{t,p}} & \LL_{t,p}
\ar[d]^{V_{q,d}\enbrak{J_{p,r}\enbrak{t\ZZ^d,\abs{\cdot}^{1/q}}}} \\
\LL_r \ar[r]_{A_{t,r}}& \LL_{t,r}}.
\end{gathered}
\end{equation}
commute.

If $T_p\colon \XX \to \LL$ is a bounded linear map such that $T_p(L_p(\mu ))=\chi_{[0,1)^d}\otimes \mu$ for all $\mu\in\YY$, then the family of operators $(T_{t,p})_{t>0}$ given by
\[
T_{t,p}= A_{t,p} \circ T_p \circ D_{1/t,p}
\]
satisfies the desired conditions. Indeed,
\[
\norm{T_{t,p}}\le \norm{A_{t,p}} \norm{T_p} \norm{D_{1/t,p}}= t^{1/q} \norm{T_p} t^{-1/q}={\norm{T_p}},
\]
given $\mu\in\YY_{t,p}$,
\begin{multline*}
T_{t,p}(L_{t,p}(\mu))
=A_{t,p}(T_p(L_p(C_{1/t,p}(\mu)))
=A_{t,p} \enpar{\chi_{[0,1)^d} \otimes C_{1/t,p}(\mu)}\\
= \chi_{[0,1)^d} \otimes B_{t,p}( (C_{1/t,p}(\mu))=\chi_{[0,1)^d} \otimes \mu,
\end{multline*}
and merging the commutative diagrams \eqref{eq:diagX} with $t^{-1}$, \eqref{eq:diagL} and
\begin{equation}\label{eq:diagZ}
\begin{gathered}
\xymatrixcolsep{6pc}
\xymatrix{\XX_p \ar[d]_{V_{q,d}\enbrak{J_{p,r}\enbrak{\RR^d,\abs{\cdot}^{1/q}}}} \ar[r]^{T_p} & \LL_{p}
\ar[d]^{V_{q,d}\enbrak{J_{p,r}\enbrak{\ZZ^d,\abs{\cdot}^{1/q}}}} \\
\XX_r \ar[r]_{T_r}& \LL_{r}}.
\end{gathered}
\end{equation}
yield the whised-for commutative diagram.

To show the existence of $T_p$, we consider the function $\tau_{p}\colon\RR^d \to \LL_p$ given by
\[
\tau_{p}(x)=\sum_{w\in\ZZ^d} \varphi_d[w,x] \otimes \delta_{p,\ZZ^d}(w).
\]
By \ref{it:1:DD} and \ref{it:21:DD}, $\tau_p(w)= \chi_{[0,1)^d} \otimes \delta_{p,\ZZ^d}(w)$ for all $w\in\ZZ^d$. Aiming to prove that $\tau_p$ is a Lipschitz function when we equip $\RR^d$ with the $q$-metric $\abs{\cdot}_1^{1/q}$, we pick $x=(x_j)_{j=1}^d$ and $y=(y_j)_{j=1}^d$ in $\RR^d$. Let $z=(z_j)_{j=1}^d$ be the unique vector in $\RR^d$ such that
\begin{enumerate}[label=(\alph*),leftmargin=*]
\item $u_j:=x_j-z_j\in\ZZ$,
\item\label{it:close} $b_j:=z_j-y_j\in[-1,1]$, and
\item $\abs{x_j-y_j}=\abs{u_j}+\abs{b_j}$
\end{enumerate}
for all $j=1$, \dots, $d$. By condition~\ref{it:close}, there exists $v\in V_d(z)\cap V_d(y)$. Set $V_d(y,z)= V_d(y)\cup V_d(z)$. By \ref{it:1:DD} and \ref{it:2:DD},
\begin{align*}
\tau_{p}(z)-\tau_{p}(y)
&=\sum_{w\in\ZZ^d}\varphi_d[w,z]\otimes\delta_{p,\ZZ^d}(w)-\sum_{w\in\ZZ^d} \varphi_d[w,y] \otimes\delta_{p,\ZZ^d}(w)\\
&-\sum_{w\in\ZZ^d} \enpar{\varphi_d[w,z] -\varphi_d[w,y]} \otimes \delta_{p,\ZZ^d}(v)\\
&=\sum_{w\in\ZZ^d}\enpar{\varphi_d[w,z] -\varphi_d[w,y]} \otimes \enpar{\delta_{p,\ZZ^d}(w) - \delta_{p,\ZZ^d}(v)}
\\
&=\sum_{w\in V_d(y,z)}\enpar{\varphi_d[w,z] -\varphi_d[w,y]} \otimes \enpar{\delta_{p,\ZZ^d}(w) - \delta_{p,\ZZ^d}(v)}.
\end{align*}
In turn, by \ref{it:1:DD} and \ref{it:3:DD}, setting $u=(u_j)_{j=1}^d$,
\begin{align*}
\tau_{p}(x)-\tau_{p}(z)
&=\sum_{w\in\ZZ^d} \varphi_{d}[w,x] \otimes \delta_{p,\ZZ^d}(w) - \sum_{w\in\ZZ^d}\varphi_{d}[w,x-u] \otimes \delta_{p,\ZZ^d}(w)\\
&=\sum_{w\in\ZZ^d} \varphi_d[w,x] \otimes\enpar{\delta_{p,\ZZ^d}(w)-\delta_{p,\ZZ^d}(w-u)}.\\
&=\sum_{w\in V_d(x)} \varphi_d[w,x] \otimes \enpar{\delta_{p,\ZZ^d}(w)-\delta_{p,\ZZ^d}(w-u)}.
\end{align*}
Consequently, by \ref{it:6:DD} and \ref{it:7:DD}, letting $b=(b_j)_{j=1}^d$,
\begin{align*}
\norm{\tau_{p}(x)-\tau_{p}(y)}_{q,\YY}^p
&=\norm{\tau_{p}(x)-\tau_{p}(z)}_{q,\YY}^p+\norm{\tau_{p}(z)-\tau_{p}(y)}_{q,\YY}^p\\
&= \abs{V_d(x)} \norm{u}_\infty^{p/q}+ \abs{V_d(y,z) \setminus \{v\}} \norm{b}_1^{p/q}\\
&= 2^d \norm{u}_\infty^{p/q}+\enpar{2^{2d}-1} \norm{b}_1^{p/q}\\
&\le \enpar{2^{2d}-1} \enpar{\norm{u}_1^{p/q} +\norm{b}_1^{p/q}}\\
&\le 2^{1-p/q} \enpar{2^{2d}-1} \enpar{\norm{u}_1 +\norm{b}_1}^{p/q}\\
&=2^{1-p/q} \enpar{2^{2d}-1} \norm{x-y}_1^{p/q}.
\end{align*}
The proof now is over once we realize that the operator $T=E[\tau_p;p]$ satisfies the desired conditions.
\end{proof}
\subsection{Towards the final steps in the resolution of Question~\ref{qt:AKEnv}}
We are now almost ready to complete the proof that answers Question~\ref{qt:AKEnv} in the positive. Prior to that, we record in a lemma an elementary idea that we will use twice.

\begin{lemma}\label{lem:Aux}
Let $0<p\le 1$. Let $\XX$ be a $p$-Banach space, $(\XX_\lambda)_{\lambda\in\Lambda}$ and $(\YY_\lambda)_{\lambda\in\Lambda}$ be families of $p$-Banach spaces, and for each $\lambda\in\Lambda$ $T_\lambda\colon\XX \to \YY_\lambda$ and $S_\lambda\colon \XX_\lambda\to \XX$ be linear operators. Suppose that
\begin{itemize}
\item $\VV:=\cup_{\lambda\in\Lambda} S_\lambda(\XX_\lambda)$ is dense in $\XX$;
\item $C:=\sup_\lambda \norm{T_\lambda} <\infty$;
\item $D:= \sup_\lambda \norm{S_\lambda} <\infty$; and
\item for all $\lambda\in\Lambda$, $T_\lambda \circ S_\lambda$ is an isomorphic embedding, and
\[
E:=\sup_\lambda \norm{(T_\lambda \circ S_\lambda)^{-1}}<\infty.
\]
\end{itemize}
Then the operator
\[
T=\colon \XX \to \enpar{\bigoplus_{\lambda\in\Lambda} \YY_\lambda}_{\ell_\infty},
\quad x \mapsto \enpar{T_\lambda(x)}_{\lambda\in\Lambda}
\]
is an isomorphic embedding.
\end{lemma}

\begin{proof}
Clearly, $T$ is bounded by $C$. If $v\in\VV$, then there are $\lambda\in\Lambda$ and $x\in\XX_\lambda$ such that $v=S_\lambda(x)$. We have
\[
\norm{v} \le D \norm{x}\le D E \norm{T_\lambda(S_\lambda(x))} =D E \norm{T_\lambda(v)} \le D E \norm{T(v)}.
\]
By continuity and density, $\norm{y} \le D E \norm{T(y)}$ for all $y\in\XX$.
\end{proof}

\begin{theorem}
Let $0<p\le q\le 1$ and $\Mt$ be a pointed $q$-metric space. Then the canonical envelope map $J_{p,q}[\Mt]$ is one-to-one.
\end{theorem}

\begin{proof}
Combining Proposition~\ref{prop:SchauderSeparates} with Theorem~\ref{thm:FDDZdq} yields the result for $(\ZZ^d,\abs{\cdot}_\infty^{1/q})$, $d\in\NN$. By dilation, it also holds for $(t \ZZ^d,\abs{\cdot}_\infty^{1/q})$, $t>0$.

Set for each $k\in\NN_0$,
\[
\YY_k[p,q,d]=\F_p\enpar{2^{-k} \ZZ^d,\abs{\cdot}_\infty^{1/q}}, \quad
\XX_k[p,q,d]=L_q\enpar{[0,1)^d, \YY_k[p,q,d]}.
\]
Also set
\[
\XX[p,q,d]= \enpar{\bigoplus_{k=1}^\infty \XX_k[p,q,d]}_{\ell_\infty}, \quad \WW[p,q,d]=\F_p\enpar{\RR^d,\abs{\cdot}_1^{1/q}}.
\]
For each $t>0$ and $0<p\le q$, let $L_{t,p}$ and $T_{t,p}$ be as in Theorem~\ref{thm:Embedding}. Define
\[
T_p:=T[p,q,d]\colon \WW[p,q,d] \to \XX[p,q,d],
\quad \mu\mapsto(T_{2^{-k},p}(\mu))_{k=0}^\infty.
\]
Since for every $ \mu\in\YY_k[p,q,d]$
\[
T_p(L_{2^{-k},p}(\mu)) =\chi_{[0,1)^d}\otimes\mu, \quad \norm{\chi_{[0,1)^d}\otimes\mu}=\norm{\mu},
\]
$T_p\circ L_{2^{-k},p}$ is an isometric embedding for all $k\in\NN_0$. The set
\[
\VV=\bigcup_{k=0}^\infty L_{2^{-k},p}( \YY_k[p,q,d])
\]
is the subspace of $\WW[p,q,d]$ consisting of all vectors supported on
\[
\Mt_d:=\bigcup_{k=1}^\infty 2^{-k} \ZZ^d.
\]
Since $\Mt_d$ is a dense subset of $\RR^d$, $\VV$ is dense in $\WW[p,q,d]$. Consequently, $T_p$ is an isomorphic embedding by Lemma~\ref{lem:Aux}. In particular, it is one-to-one.

Since $I_k[p,q,d]:=J_{p,q}[2^{-k}\ZZ^d, \abs{\cdot}_\infty^{1/q}]$ is one-to-one, the operator
\[
I[p,q,d]:=\enpar{V_{q,d}\enbrak{I_k[p,q,d]}}_{k=0}^\infty\colon \XX[p,q,d] \to \XX[q,q,d]
\]
is one-to-one. A routine check shows that the diagram
\[
\xymatrixcolsep{5pc}
\xymatrix{\WW[p,q,d] \ar[r]^{T[p,q,d]} \ar[d]_{J_{p,q}\enbrak{\RR^d,\abs{\cdot}^{1/q}}} &\XX[p,q,d] \ar[d]^{I[p,q,d]} \\
\WW[q,q,d] \ar[r]_{T[q,q,d]} & \XX[q,q,d]}
\]
conmutes. Therefore, $J_{p,q}[\RR^d,\abs{\cdot}^{1/q}]$ is one-to-one, that is, the result holds for $(\RR^d,\abs{\cdot}^{1/q})$.

Next, we prove the result for a general $q$-metric space $\Mt$. Fix $C>1$. By Theorem~\ref{thm:metricembedding}, $\Mt$ isometrically embeds into $(\XX, \norm{\cdot}^{1/q})$, where $\XX$ is a Banach space for which there are an increasing directed set $(\VV_\lambda)_{\lambda\in \Lambda}$ consisting of finite-dimensional subspaces of $\XX$ such that $\VV:=\cup_{\lambda\in \Lambda} \VV_\lambda$ is dense in $\XX$, and for each $\lambda\in \Lambda$ a linear projection $P_\lambda\colon \XX\to\XX$ with $P_\lambda(\XX)=\VV_\lambda$ and $\norm{P_\lambda}\le C$.

Let $Q_\lambda[p,q]$ be the linearization at level $p$ of $P_\lambda$ regarded a Lipschitz map from $(\XX, \norm{\cdot}^{1/q})$ into $(\VV_\lambda,\norm{\cdot}^{1/q})$. Similarly, let $L_\lambda[p,q]$ be the linearization at level $p$ of the embedding of $(\VV_\lambda,\norm{\cdot}^{1/q})$ into $(\XX, \norm{\cdot}^{1/q})$. Since $\norm{Q_\lambda[p,q]}\le C$,
$\norm{L_\lambda[p,q]}\le 1$, $Q_\lambda[p,q]\circ L_\lambda[p,q]$ is the identity map on
\[
\YY_\lambda[p,q]:=\F_p(\VV_\lambda,\norm{\cdot}^{1/q}),
\]
and $\cup_{\lambda\in\Lambda} L_\lambda[p,q](\YY_\lambda[p,q])$ is the subspace of
\[
\WW[p,q]:=\F_p(\XX, \norm{\cdot}^{1/q})
\]
consisting of all vectors supported on $\VV$, the linear map
\[
Q[p,q] \colon \WW[p,q] \mapsto \YY[p,q]:=\enpar{\bigoplus_{\lambda\in \Lambda} \YY_\lambda[p,q]}_{\ell_\infty},
\quad \mu\mapsto (Q_\lambda[p,q](\mu))_{\lambda\in \Lambda}
\]
is an isomorphic embedding by Lemma~\ref{lem:Aux}. In particular, $Q[p,q]$ is one-to-one. In turn, the result about Euclidean spaces gives that the linear contraction
\[
S[p,q] \colon \YY[p,q] \to \YY[q,q],
\quad (\mu_\lambda)_{\lambda\in\Lambda} \mapsto (J_{p,q}[(\VV_\lambda, \norm{\cdot}^{1/q})](\mu_\lambda))_{\lambda\in\Lambda},
\]
is one-to-one. Finally, $L[p,q]:=L[\Mt,(\XX,\norm{\cdot}^{1/q});p]$ is one-to-one by Theorem~\ref{thm:LIEP}. A straightforward check yields that the diagram
\[
\xymatrixcolsep{5pc}\xymatrix{\F_p(\Mt) \ar[r]^{L[p,q]} \ar[d]_{J_{p,q}[\Mt]}
& \WW[p,q] \ar[r]^{Q[p,q]} \ar[d]_{J_{p,q}[\XX]}
&\YY[p,q] \ar[d]^{S[p,q]} \\
\F_q(\Mt) \ar[r]_{L[q,q]}
&\WW[q,q] \ar[r]_{Q[q,q]}
& \YY[q,q]}
\]
commutes. Consequently, $J_{p,q}[\Mt]$ is one-to-one.
\end{proof}
\subsection*{Acknowledgement}
The authors would like to thank Prof.\@ Marek C\'{u}th for his constructive feedback and fine comments, which helped improve the manuscript.
\subsection*{Conflict of interest}
The authors have no competing interests to declare that are relevant to the content of this article.
\subsection*{Data Availability}
Data sharing does not apply to this article, as no datasets were generated or analysed during the current study.

\end{document}